\documentclass[a4paper,10pt]{amsart}

\usepackage[english]{babel}
\usepackage{amsmath,amssymb,amsthm}
\usepackage{hyperref}
\usepackage{enumerate}
\usepackage[centering]{geometry}
\usepackage{xcolor}

\newcommand{\R}{\ensuremath{\mathbb{R}}}

\newcommand{\SR}{\ensuremath{\mathbb S^3 \times \mathbb R}}

\newcommand{\HR}{\ensuremath{\mathbb H^3 \times \mathbb R}}

\newcommand{\Hi}{\mathbb{H}}
\newcommand{\Sf}{\mathbb{S}}
\newcommand{\Ri}{\mathbb{R}}

\newcommand{\Qi}{\mathbb{Q}_{\varepsilon}}

\newtheorem{theorem}{Theorem}  
\newtheorem*{theorem*}{Main Theorem}          

\newtheorem{lemma}{Lemma}

\theoremstyle{definition}

\title[]{Hypersurfaces of $\SR$ and $\HR$ \\ with constant principal curvatures}

\thanks{F. Manfio was supported by FAPESP, grant 2022/16097-2. J. B. M. dos Santos was supported by Capes and CNPq. J. P. dos Santos was supported by CNPq 315614/2021-8 and FAPDF 0193.001346/2016. J. Van der Veken was supported by the Research Foundation - Flanders (FWO) and the Fonds de la Recherche Scientifique (FNRS) under EOS project G0I2222N and by the KU Leuven Research Fund under project 3E210539. J. P. dos Santos and J. Van der Veken were also supported by Young Researchers Summer Programme 2018 - USP}

\author{F. Manfio, J. B. M. dos Santos, J. P. dos Santos \and J. Van der Veken}

\address{Fernando Manfio - Instituto de Ci\^encias Matem\'aticas e 
Computa\c c\~ao, Universidade de S\~ao Paulo, 13566-590, S\~ao 
Carlos, Brazil}
\email{manfio@icmc.usp.br}

\address{Jo\~ao Batista Marques dos Santos - Centro Multidisciplinar, Universidade Federal do Acre, 69980000, Cruzeiro do Sul-AC, Brazil}
 \email{joao.batista.m.s@ufac.br}

\address{Jo\~ao Paulo dos Santos - Departamento de Matem\'atica, 
Universidade de Bras\'ilia, 70910-900, Bras\'ilia-DF, Brazil}
\email{joaopsantos@unb.br}

\address{Joeri Van der Veken, KU\ Leuven, Department of Mathematics, 
Celestijnenlaan 200B - Box 2400, 3001 Leuven, Belgium}
\email{joeri.vanderveken@kuleuven.be}

\date{\today}

\begin{document}


\begin{abstract}
We classify the hypersurfaces of $\Qi^3\times\R$ with three distinct constant principal curvatures, where $\varepsilon \in \{1,-1\}$ and $\Qi^3$ denotes the unit sphere $\Sf^3$ if $\varepsilon = 1$, whereas it denotes the hyperbolic space $\Hi^3$ if $\varepsilon = -1$. We show that they are cylinders over isoparametric surfaces in $\Qi^3$, filling an intriguing gap in the existing literature. We also prove that the hypersurfaces with constant principal curvatures of $\Qi^3\times\R$ are isoparametric. Furthermore, we provide the complete classification of the extrinsically homogeneous hypersurfaces in $\Qi^3 \times \Ri$.
\end{abstract}
	
\maketitle

\section{Introduction}

A non-constant smooth function $f: \widetilde{M}\to\R$, defined on a Riemannian manifold $\widetilde{M}$, is called isoparametric if there exist smooth functions $a,b:\R\to\R$ such
that
\begin{equation*}
|| \nabla f ||^2 = a(f) \quad \mbox{and} \quad \Delta f = b(f),
\end{equation*} 
and the regular level sets of $f$ are called isoparametric hypersurfaces of $\widetilde{M}$. Equivalently, a hypersurface $\Sigma$ of a Riemannian manifold $\widetilde{M}$ is isoparametric if $\Sigma$ and the nearby parallel hypersurfaces have constant mean curvature. 

Cartan showed that when the ambient space is a space form, a hypersurface is isoparametric if and only if its principal curvatures are constant \cite{isoCartan}. In ambient spaces of nonconstant curvature, this equivalence does not necessarily occur. On the one hand, we have in \cite{diaz2010inhomogeneous, diaz2013isoparametric, ge2015filtration, WangExIso} examples of isoparametric hypersurfaces that do not have constant principal curvatures. On the other hand, examples of hypersurfaces with constant principal curvatures that are not isoparametric can be found in  \cite{isoparametric-mcf, RVazquezExIso}. Nevertheless, the equivalence may still hold in ambient spaces other than space forms, as we can see in \cite{dominguez-manzano} for homogeneous 3-spaces, and in \cite{gao-2} for products of 2-dimensional space forms, concluding the classification started in \cite{s2r2Batalla, santos2022isoparametric}. For more details on isoparametric hypersurfaces, we refer to \cite{chi,notas-miguel, survey-thorbergsson}. 

In this paper, we consider isoparametric hypersurfaces and hypersurfaces with constant principal curvatures of the product spaces $\Qi^3\times\R$, where $\varepsilon \in \{1,-1\}$ and $\Qi^3$ denotes the unit sphere $\Sf^3$ if $\varepsilon = 1$, whereas it denotes the hyperbolic space $\Hi^3$ if $\varepsilon = -1$. Our goal is to classify the hypersurfaces of $\Qi^3\times\R$ with three distinct constant principal curvatures and provide the main consequences of this classification. The target dimension and number of distinct principal curvatures are motivated by the classification given by Chaves and Santos \cite{Rosa-and-Eliane} of the hypersurfaces of $\Qi^n\times\R$, $n \geq 2$,  with $g$ distinct constant principal curvatures, $g \in \{1,2,3\}$,  and $n\geq4$ if $g=3$. Chaves and Santos also proved that such hypersurfaces are isoparametric in those spaces. In this context, our first main result provides the classification of the hypersurfaces in $\Qi^3\times\R$ with $g=3$ constant principal curvatures. 

\begin{theorem} \label{theorem-classification}
Let $\Sigma$ be a connected hypersurface of $\Qi^3\times\R$ with three distinct constant principal curvatures. Then $\Sigma$ is an open part of the following hypersurfaces:
\begin{itemize}
\item [(a)] $\Sf^1(r_1)\times\Sf^1(r_2)\times\R$, when $\varepsilon =1$;
\item [(b)] $\Hi^1(r_1)\times\Sf^1(r_2)\times\R$, when $\varepsilon =-1$,
\end{itemize}
where $r_1 \neq r_2$ and $r_1^2 + \varepsilon r_2^2 = 1$. The principal curvatures of $\Sigma$ are $0$, $\dfrac{\varepsilon r_2}{r_1}$ and $-\dfrac{r_1}{r_2}$.
\end{theorem}

Theorem \ref{theorem-classification} completes the classification of the hypersurfaces of $\Qi^n\times\R$ that have $g$ distinct constant principal curvatures, $g \in \{1,2,3\}$, started in \cite[Theorem 6.1]{Rosa-and-Eliane}, by providing the missing piece given by the case $n=g=3$. It is worth mentioning that the problem for $g \geq 4$ remains
open. A fundamental step in the proof of Theorem \ref{theorem-classification} consists in proving that the normal direction of a connected hypersurface of $\Qi^3\times\R$ with three distinct constant principal curvatures makes a constant angle $\theta$ with the factor $\R$ (Lemma \ref{lemma-characterization}). This property is also a key step for the proof of our next result.

\begin{theorem}\label{Corollary}
Let $\Sigma$ be a connected hypersurface of $\Qi^3 \times \Ri$ with constant principal curvatures. Then  $\Sigma$ is isoparametric.
\end{theorem}

As consequence of Theorem \ref{theorem-classification} and \cite[Theorem 6.1]{Rosa-and-Eliane}, we have the classification of the extrinsically homogeneous hypersurfaces in $\Qi^3 \times \Ri$. Recall that a hypersurface $\Sigma$ in a Riemannian manifold $\widetilde{M}$ is an extrinsically homogeneous hypersurface if $\Sigma$ is an orbit of codimension one of a subgroup of the isometry group of $\widetilde{M}$. The classification is given in our third main result.

\begin{theorem} \label{corollary-2}
Let $\Sigma$ be an extrinsically homogeneous hypersurface in $\Qi^3 \times \Ri$. Then $\Sigma$ is one of the following:
\begin{enumerate}[(i)]
\item a horizontal slice $\Qi^3 \times \{t_0\}$;
\item a vertical cylinder $\Sigma_0 \times \Ri$, where $\Sigma_0$ is either 
\begin{enumerate}[(a)]
\item a totally geodesic surface in $\Qi^3$;
\item a totally umbilical surface in $\Qi^3$;
\item $\Sf^1(r_1)\times\Sf^1(r_2) \subset \Sf^3$, with $r_1^2+r_2^2 =1$;
\item $\Hi^1(r_1)\times\Sf^1(r_2) \subset \Hi^3$, with $r_1^2-r_2^2 =1$;
\end{enumerate}
\item $\{(\textnormal{exp}_p (s \eta(p)), B s) \ | \ p \in \mathcal H, \ s \in \R\} \subset \Hi^3 \times \Ri$, where $\mathcal{H} \subset \Hi^3$ is a horosphere with unit normal vector field $\eta$ and $B$ is a real constant.
\end{enumerate}
\end{theorem}

The hypersurface given in item (iii) of Theorem \ref{corollary-2} is an entire graph over a family of parallel horospheres in $\Hi^3$. If we use the half-space model for $\Hi^3$, this hypersurface is given by $\{(u,\,v,\,w, B \log (w)) \, | \, u,v\in\R, \, w>0 \} \subset \Hi^3 \times \Ri$. The surface $\{(u,\,w, B \log(w)) \, | \, u \in \R, \, w>0 \}$ in $\Hi^2 \times \Ri$ is called a parabolic helicoid and it is precisely one of the homogeneous surfaces in $\Hi^2 \times \Ri$, as classified by Dom\'inguez-V\'azquez and Manzano in \cite{dominguez-manzano}.

\section{The main Lemma}

In this section we establish our main Lemma, which states that a connected hypersurface of $\Qi^3 \times \Ri$ with distinct constant principal curvatures has constant angle function (Lemma \ref{lemma-characterization}). Using this characterization, we will then prove Theorem \ref{theorem-classification} and Theorem \ref{Corollary} in the next section.

Given a hypersurface $\Sigma$ of $\Qi^{n}\times\R$ with a unit normal vector field $N$, let $S$ be the shape operator of $\Sigma$  and let $\nabla$ be the Levi-Civita connection of $\Sigma$. Since $\partial_{t}$ is a unit vector field, it follows that $||T||^2 + \cos^2\theta=1$. Moreover, since $\partial_t$ is parallel in $\Qi^n\times\Ri$, by differentiating it and using the decomposition $\partial_{t} = T + \cos\theta\,N$, we obtain that
\begin{align}
& \nabla_{X}T = \cos\theta \, SX, \label{T-deriv} \\
& X(\cos\theta) = -\langle X,ST\rangle \label{Xcos}
\end{align}
for all $X\in T\Sigma$.

Let $R$ be the curvature tensor of $\Sigma$ defined as $R(X,Y)Z =\nabla_X \nabla_Y Z - \nabla_Y \nabla_X Z - \nabla_{[X,Y]}Z$, then the Gauss and Codazzi equations are given by
\begin{equation} \label{gauss}
\begin{aligned}
\langle R(X,Y)Z,W\rangle = \ & \varepsilon(\langle X,W\rangle
\langle Y,Z\rangle-\langle X,Z\rangle\langle Y,W\rangle \\
&+ \langle X,T\rangle\langle Z,T\rangle\langle Y,W\rangle+
\langle Y,T\rangle\langle W,T\rangle\langle X,Z\rangle \\
& - \langle Y,T\rangle\langle Z,T\rangle\langle X,W\rangle-
\langle X,T\rangle\langle W,T\rangle\langle Y,Z\rangle) \\
& + \langle SX,W\rangle\langle SY,Z\rangle-\langle SX,Z\rangle
\langle SY,W\rangle
\end{aligned}
\end{equation}
and
\begin{equation}\label{codazzi}
\nabla_{X}(SY)-\nabla_{Y}(SX)-S[X,Y] = 
\varepsilon\cos\theta(\langle Y,T\rangle X-\langle X,T\rangle Y)
\end{equation}
respectively, where $X,Y,Z,W\in T\Sigma$.

\begin{lemma}\label{lemma-equations}
Let $\Sigma$ be a hypersurface of $\Qi^3 \times \Ri$ with three distinct constant principal curvatures and angle function $\theta$. Let $\{e_1,e_2,e_3\}$ be a local frame of orthonormal principal directions defined on $\Omega \subset \Sigma$ and denote by $\mu_{i}$ the principal curvature associated to $e_i$ for $i \in \{1,2,3\}$. If we define $b_1, \, b_2, \, b_3 : \Omega \rightarrow \R$ by
\begin{equation} \label{T-decomp}
T = \displaystyle \sum_{i=1}^3 b_i e_i, 
\end{equation}
then there exists a function $A: \Omega \to \R$ such that, for all distinct indices $l,\,m,n \in \left\{1,\,2,\,3 \right\}$,
\begin{equation}
e_n(\cos\theta) = - \mu_n b_n, \label{cos-deriv}
\end{equation}
\begin{equation}
e_n(b_n) = \left[ \mu_n - \varepsilon \left( \dfrac{b_l^2}{\mu_l-\mu_n} + \dfrac{b_m^2}{\mu_m-\mu_n}  \right) \right] \cos\theta, \label{en-bn} 
\end{equation}
\begin{equation}
\dfrac{2A^2\cos^2\theta}{(\mu_m-\mu_l)(\mu_l-\mu_n)} + \dfrac{2(b_n^2+b_m^2)\cos^2\theta}{(\mu_n-\mu_m)^2} +\dfrac{\varepsilon(\mu_nb_n^2-\mu_mb_m^2)}{\mu_m-\mu_n} + 2 \varepsilon \cos^2\theta +  \varepsilon b_l^2 + \mu_m \mu_n = 0. \label{gauss-mnmn}
\end{equation}

\end{lemma}

\begin{proof}
Equation \eqref{cos-deriv} is a direct consequence of \eqref{Xcos}. If $\nabla$ denotes the Levi-Civita connection of $\Sigma$, then
\begin{equation}
\nabla_{e_i} e_j = \displaystyle \sum_{k=1}^3 \omega_j^k(e_i) e_k, \label{levi-decomp}
\end{equation}
where $\omega_j^k$ are the connection forms of $\Sigma$. In what follows, $i$, $j$ and $k$ will denote arbitrary indices in $\{1,2,3\}$, while $l$, $m$ and $n$ will denote distinct indices in $\{1,2,3\}$, as in the formulation of the lemma.
		
The Codazzi equation \eqref{codazzi} for $X=e_i$ and $Y=e_j$, with $i \neq j$, together with \eqref{T-decomp} and \eqref{levi-decomp}, gives
$$
\displaystyle \sum_{k=1}^3 \left[\left(\mu_j-\mu_k\right) \omega^k_j(e_i)-\left(\mu_i-\mu_k\right) \omega^k_i(e_j)  \right] e_k = \varepsilon \cos\theta \left(b_j e_i - b_i e_j\right).
$$
Considering the coefficients of each principal direction we conclude that
$$
\begin{array}{rcl}
(\mu_m - \mu_n) \omega_m^n(e_n) &=& \varepsilon b_m\cos\theta,  \\
\left(\mu_m-\mu_l\right) \omega^l_m(e_n)&=&\left(\mu_n-\mu_l\right) \omega^l_n(e_m).
\end{array}
$$
Consequently, 
\begin{equation}
\omega_m^n(e_n) = \dfrac{\varepsilon b_m\cos\theta}{\mu_m - \mu_n}, \label{omega-mn} 
\end{equation}
and there is a smooth function $A : \Omega \rightarrow \mathbb{R}$ such that
\begin{equation}
\omega^l_m(e_n)=\dfrac{\varepsilon A \cos\theta}{\mu_m-\mu_l}. \label{omega-mnl}
\end{equation}

To obtain equation \eqref{en-bn} we first observe that, from equations \eqref{T-deriv} and \eqref{T-decomp}, we have
$$
\mu_n\cos\theta \, e_n = \displaystyle \sum_{i=1}^3 \left[ e_n(b_i) + \sum_{k=1}^3 b_k \omega_k^i(e_n) \right] e_i.
$$
Then, using equations \eqref{omega-mn} and \eqref{omega-mnl}, we conclude \eqref{en-bn}. 

To obtain \eqref{gauss-mnmn}, we will use the Gauss equation \eqref{gauss}. On the one hand, note that the bracket between principal directions is given by
$$
\left[ e_m, \, e_n \right] = \displaystyle \sum_{k=1}^3 \left( \omega_n^k(e_m)-\omega_m^k(e_n) \right) e_k
= \omega_n^m(e_m) e_m - \omega_m^n(e_n) e_n + \left( \omega_n^l(e_m)-\omega_m^l(e_n) \right) e_l, \\
$$
which implies
\begin{align*}
\langle R(e_m, e_n) e_m, e_n \rangle = \ & \omega_m^l(e_n) \omega_l^n(e_m) + e_m(\omega_m^n(e_n))-\omega_m^l(e_m)\omega_l^n(e_n)-e_n(\omega_m^n(e_m)) \\
& +(\omega_n^m(e_m))^2 + (\omega_m^n(e_n))^2-\omega_n^l(e_m)\omega_m^n(e_l) + \omega_m^l(e_n) \omega_m^n(e_l) \\
= \ & \dfrac{2A^2\cos^2\theta}{(\mu_m-\mu_l)(\mu_l-\mu_n)} + \dfrac{\varepsilon(\mu_nb_n^2-\mu_mb_m^2)}{\mu_m-\mu_n} + \varepsilon \cos^2\theta + \dfrac{2 (b_n^2+b_m^2)\cos^2\theta}{(\mu_n-\mu_m)^2},
\end{align*}
where we used \eqref{cos-deriv}, \eqref{en-bn}, \eqref{omega-mn} and \eqref{omega-mnl}. On the other hand, by Gauss equation \eqref{gauss}, we have $$
\langle R(e_m,e_n)e_m,e_n \rangle = \varepsilon (b_m^2+b_n^2-1)-\mu_m \mu_n.$$
By combining the two previous equations, we can conclude \eqref{gauss-mnmn}.
\end{proof}

\begin{lemma}\label{lemma-expressions-bi}
Let $\Sigma$ be a hypersurface of $\Qi^3 \times \Ri$ with three distinct constant principal curvatures and angle function $\theta$. Let $b_i : \Omega \subset \Sigma \rightarrow \mathbb{R}$ be the functions defined by \eqref{T-decomp} and set $t=\cos^2 \theta$. Then
\begin{align}\label{b_i^2}
b_i^2(t) = \dfrac{p_i(t)}{q(t)},
\end{align}
where $p_i(t)$ and $q(t)$ are polynomials in the variable $t$, given by
\begin{align*}
p_1(t) &= -2(\mu_1 - \mu_2)^4(3(\mu_1 - \mu_3)^2 + 3(\mu_2 - \mu_3)^2 + (\mu_1 - \mu_2)^2)\,t^3 + Q_1(t), \\
p_2(t) &= 12(\mu_1 - \mu_2)^4(\mu_1 - \mu_3)(\mu_2 - \mu_3)\,t^3 + Q_2(t), \\
p_3(t) &= -2(\mu_2 - \mu_3)(\mu_1 - \mu_2)^3(3(\mu_1 - \mu_2)^2 + 3(\mu_1 - \mu_3)^2 + (\mu_2 - \mu_3)^2)\,t^3 + Q_3(t), \\
q(t) &= 2(\mu_1 - \mu_3)(\mu_1 - \mu_2)^3((\mu_1 - \mu_3)^2 + 3(\mu_1 - \mu_2)^2 + 3(\mu_2 - \mu_3)^2)\,t^2 + L(t),
\end{align*}
where $Q_i$ are quadratic polynomials and $L$ is a degree one polynomial.
\end{lemma}

\begin{proof}
After setting $t=\cos^2\theta$, we obtain the following equations from \eqref{gauss-mnmn}:
\begin{align}\label{gauss-1212}
\dfrac{2A^2t}{(\mu_3-\mu_1)(\mu_3-\mu_2)} = \dfrac{\varepsilon(\mu_1b_1^2-\mu_2b_2^2)}{\mu_2-\mu_1} +  \dfrac{2(b_1^2+b_2^2)t}{(\mu_2-\mu_1)^2} + 2\varepsilon t + \mu_1\mu_2 + \varepsilon b_3^2,
\end{align}
\begin{align}\label{gauss-2323}
\dfrac{2A^2t}{(\mu_2-\mu_1)(\mu_3-\mu_1)} = \dfrac{\varepsilon(\mu_2b_2^2-\mu_3b_3^2)}{\mu_3-\mu_2} +  \dfrac{2 (b_2^2+b_3^2)t}{(\mu_3-\mu_2)^2} + 2\varepsilon t + \mu_2\mu_3 + \varepsilon b_1^2,
\end{align}
\begin{align}\label{gauss-1313}
\dfrac{2A^2t}{(\mu_2-\mu_1)(\mu_3-\mu_2)} = \dfrac{\varepsilon(\mu_1b_1^2-\mu_3b_3^2)}{\mu_3-\mu_1} +  \dfrac{2 (b_1^2+b_3^2)t}{(\mu_3-\mu_1)^2} + 2\varepsilon t + \mu_1\mu_3 + \varepsilon b_2^2.
\end{align}
In what follows, we will eliminate $A$ from the above equations. By solving \eqref{gauss-1212} for $A$ and substituting the result in \eqref{gauss-2323} and \eqref{gauss-1313}, we get the following equations in the variables $b_1^2$, $b_2^2$ and $b_3^2$:
\begin{equation}\label{system-b1b2b3-1}
\begin{split}
&\left[\dfrac{\varepsilon\mu_1(\mu_3 - \mu_2)}{\mu_2 - \mu_1} + \dfrac{2(\mu_3 - \mu_2)t}{(\mu_2 - \mu_1)^2} - \varepsilon(\mu_2 - \mu_1) \right]b_1^2 \\
& + \left[\dfrac{\varepsilon\mu_2(\mu_2 - \mu_3)}{\mu_2 - \mu_1} - \dfrac{\varepsilon\mu_2(\mu_2 - \mu_1)}{\mu_3 - \mu_2} + \dfrac{2(\mu_3 - \mu_2)t}{(\mu_2 - \mu_1)^2} - \dfrac{2(\mu_2 - \mu_1)t}{(\mu_3 - \mu_2)^2}\right]b_2^2 \\
& + \left[\dfrac{\varepsilon\mu_3(\mu_2 - \mu_1)}{\mu_3 - \mu_2} - \dfrac{2(\mu_2 - \mu_1)t}{(\mu_3 - \mu_2)^2} + \varepsilon(\mu_3 - \mu_2)\right]b_3^2 \\
& + 2\varepsilon(\mu_3 - \mu_2)t - 2\varepsilon(\mu_2 - \mu_1)t + \mu_1\mu_2(\mu_3 - \mu_2) - \mu_2\mu_3(\mu_2 - \mu_1) = 0,
\end{split}
\end{equation}
\begin{equation}\label{system-b1b2b3-2}
\begin{split}
&\left[\dfrac{\varepsilon\mu_1(\mu_3 - \mu_1)}{\mu_2 - \mu_1} - \dfrac{\varepsilon\mu_1(\mu_2 - \mu_1)}{\mu_3 - \mu_1} + \dfrac{2(\mu_3 - \mu_1)t}{(\mu_2 - \mu_1)^2} - \dfrac{2(\mu_2 - \mu_1)t}{(\mu_3 - \mu_1)^2}\right]b_1^2 \\
& + \left[\dfrac{\varepsilon\mu_2(\mu_1 - \mu_3)}{\mu_2 - \mu_1} + \dfrac{2(\mu_3 - \mu_1)t}{(\mu_2 - \mu_1)^2} - \varepsilon(\mu_2 - \mu_1)\right]b_2^2 \\
& + \left[\dfrac{\varepsilon\mu_3(\mu_2 - \mu_1)}{\mu_3 - \mu_1} - \dfrac{2(\mu_2 - \mu_1)t}{(\mu_3 - \mu_1)^2} + \varepsilon(\mu_3 - \mu_1)\right]b_3^2 \\
& + 2\varepsilon(\mu_3 - \mu_1)t - 2\varepsilon(\mu_2 - \mu_1)t + \mu_1\mu_2(\mu_3 - \mu_1) - \mu_1\mu_3(\mu_2 - \mu_1) = 0.
\end{split}
\end{equation}
Since $||T||^2 + t = 1$, we have $b_3^2 = 1 - t - b_1^2 - b_2^2$. Substituting this in \eqref{system-b1b2b3-1} and \eqref{system-b1b2b3-2} yields
\begin{align*}
\begin{split}
&\left[\dfrac{\varepsilon\mu_1(\mu_3 - \mu_2)}{\mu_2 - \mu_1} - \dfrac{\varepsilon\mu_3(\mu_2 - \mu_1)}{\mu_3 - \mu_2} + \dfrac{2(\mu_3 - \mu_2)t}{(\mu_2 - \mu_1)^2} + \dfrac{2(\mu_2 - \mu_1)t}{(\mu_3 - \mu_2)^2} + \varepsilon(\mu_1 - \mu_3)\right]b_1^2 \\
& + \left[\dfrac{\varepsilon\mu_2(\mu_2 - \mu_3)}{\mu_2 - \mu_1} - \dfrac{\varepsilon\mu_2(\mu_2 - \mu_1)}{\mu_3 - \mu_2} - \dfrac{\varepsilon\mu_3(\mu_2 - \mu_1)}{\mu_3 - \mu_2} + \dfrac{2(\mu_3 - \mu_2)t}{(\mu_2 - \mu_1)^2} - \varepsilon(\mu_3 - \mu_2)\right]b_2^2 \\
& + \dfrac{2(\mu_2 - \mu_1)t^2}{(\mu_3 - \mu_2)^2} + \left(\dfrac{2(\mu_1 - \mu_2)}{(\mu_3 - \mu_2)^2} + \dfrac{\varepsilon\mu_3(\mu_1 - \mu_2)}{\mu_3 - \mu_2} + \varepsilon(\mu_3 - \mu_2) + 2\varepsilon(\mu_1 - \mu_2)\right)t \\
& + \dfrac{\varepsilon\mu_3(\mu_2 - \mu_1)}{\mu_3 - \mu_2} + \mu_1\mu_2(\mu_3 - \mu_2) + \mu_2\mu_3(\mu_1 - \mu_2) + \varepsilon(\mu_3 - \mu_2) = 0,
\end{split}
\end{align*}
\begin{align*}
\begin{split}
&\left[\dfrac{\varepsilon\mu_1(\mu_3 - \mu_1)}{\mu_2 - \mu_1} - \dfrac{\varepsilon\mu_1(\mu_2 - \mu_1)}{\mu_3 - \mu_1} - \dfrac{\varepsilon\mu_3(\mu_2 - \mu_1)}{\mu_3 - \mu_1} + \dfrac{2(\mu_3 - \mu_1)t}{(\mu_2 - \mu_1)^2} - \varepsilon(\mu_3 - \mu_1)\right]b_1^2 \\
& + \biggr[\dfrac{\varepsilon\mu_2(\mu_1 - \mu_3)}{\mu_2 - \mu_1} - \dfrac{\varepsilon\mu_3(\mu_2 - \mu_1)}{\mu_3 - \mu_1} + \dfrac{2(\mu_3 - \mu_1)t}{(\mu_2 - \mu_1)^2} + \dfrac{2(\mu_2 - \mu_1)t}{(\mu_3 - \mu_1)^2} \\
&- \varepsilon(\mu_3 - \mu_1) - \varepsilon(\mu_2 - \mu_1)\biggr]b_2^2 + \dfrac{2(\mu_2 - \mu_1)t^2}{(\mu_3 - \mu_1)^2} \\
& + \left(\dfrac{2(\mu_1 - \mu_2)}{(\mu_3 - \mu_1)^2} + \dfrac{\varepsilon\mu_3(\mu_1 - \mu_2)}{\mu_3 - \mu_1} + \varepsilon(\mu_3 - \mu_1) + 2\varepsilon(\mu_1 - \mu_2)\right)t \\
& + \dfrac{\varepsilon\mu_3(\mu_2 - \mu_1)}{\mu_3 - \mu_1} + \mu_1\mu_2(\mu_3 - \mu_1) + \mu_1\mu_3(\mu_1 - \mu_2) + \varepsilon(\mu_3 - \mu_1) = 0,
\end{split}
\end{align*}
that is, we obtain the following system in the  variables $b_1^2$ and $b_2^2$:
\begin{align}\label{linearsystemb1b2-1}
\begin{split}
&\left[\dfrac{2(\mu_3 - \mu_2)^3t + 2(\mu_2 - \mu_1)^3t + C_{11}}{(\mu_2 - \mu_1)^2(\mu_3 - \mu_2)^2}\right]b_1^2 + \left[\dfrac{2(\mu_3 - \mu_2)t + C_{12}}{(\mu_2 - \mu_1)^2}\right]b_2^2 \\
& + \dfrac{2(\mu_2 - \mu_1)t^2}{(\mu_3 - \mu_2)^2} + L_1(t) = 0,
\end{split}
\end{align}
\begin{align}\label{linearsystemb1b2-2}
\begin{split}
&\left[\dfrac{2(\mu_3 - \mu_1)t + C_{21}}{(\mu_2 - \mu_1)^2}\right]b_1^2 + \left[\dfrac{2(\mu_3 - \mu_1)^3t + 2(\mu_2 - \mu_1)^3t + C_{22}}{(\mu_2 - \mu_1)^2(\mu_3 - \mu_1)^2}\right]b_2^2 \\
& + \dfrac{2(\mu_2 - \mu_1)t^2}{(\mu_3 - \mu_1)^2} + L_2(t)=0,
\end{split}
\end{align}
where $C_{11}$, $C_{12}$, $C_{21}$ and $C_{22}$ are expressions depending on $\mu_1, \, \mu_2, \, \mu_3, \, \varepsilon$, but not on $t$, and $L_1(t)$ and $L_2(t)$ are degree one polynomials in the variable $t$. Therefore, $b_1^2$ and $b_2^2$ are solutions of the linear system
\begin{align}
M
\left(
\begin{array}{c}
b_1^2 \\
b_2^2
\end{array}
\right)
=
\left(
\begin{array}{c}
\dfrac{2(\mu_1 - \mu_2)t^2}{(\mu_3 - \mu_2)^2} - L_1(t) \\
\dfrac{2(\mu_1 - \mu_2)t^2}{(\mu_3 - \mu_1)^2} - L_2(t)
\end{array}
\right),
\end{align}
where $M$ is the matrix
\begin{align*}
M = 
\left(
\begin{array}{cc}
\dfrac{2(\mu_3 - \mu_2)^3t + 2(\mu_2 - \mu_1)^3t + C_{11}}{(\mu_2 - \mu_1)^2(\mu_3 - \mu_2)^2} & \dfrac{2(\mu_3 - \mu_2)t + C_{12}}{(\mu_2 - \mu_1)^2} \\
\dfrac{2(\mu_3 - \mu_1)t + C_{21}}{(\mu_2 - \mu_1)^2} & \dfrac{2(\mu_3 - \mu_1)^3t + 2(\mu_2 - \mu_1)^3t + C_{22}}{(\mu_2 - \mu_1)^2(\mu_3 - \mu_1)^2}
\end{array}
\right).
\end{align*}
Thus, we have
\begin{align*}
b_1^2 &= \dfrac{1}{\det M}\biggr[\left(\dfrac{2(\mu_3 - \mu_1)^3t + 2(\mu_2 - \mu_1)^3t + C_{22}}{(\mu_2 - \mu_1)^2(\mu_3 - \mu_1)^2}\right) \left(\dfrac{2(\mu_1 - \mu_2)t^2 - (\mu_3 - \mu_2)^2L_1(t)}{(\mu_3 - \mu_2)^2}\right) \\
&\quad \quad \quad \quad \quad - \left(\dfrac{2(\mu_3 - \mu_2)t + C_{12}}{(\mu_2 - \mu_1)^2}\right)\left(\dfrac{2(\mu_1 - \mu_2)t^2 - (\mu_3 - \mu_1)^2L_2(t)}{(\mu_3 - \mu_1)^2}\right)\biggr] \\
&= \dfrac{1}{\det M}\biggr[\dfrac{4(\mu_1 - \mu_2)\left((\mu_3 - \mu_1)^3 + (\mu_2 - \mu_1)^3\right)t^3}{(\mu_2 - \mu_1)^2(\mu_3 - \mu_1)^2(\mu_3 - \mu_2)^2} + \dfrac{4(\mu_1 - \mu_2)(\mu_2 - \mu_3)t^3}{(\mu_2 - \mu_1)^2(\mu_3 - \mu_1)^2} \\
& \quad + \dfrac{2C_{22}(\mu_1 - \mu_2)t^2 - 2(\mu_3 - \mu_2)^2\left((\mu_3 - \mu_1)^3 + (\mu_2 - \mu_1)^3\right)tL_1(t) - C_{22}(\mu_3 - \mu_2)^2L_1(t)}{(\mu_2 - \mu_1)^2(\mu_3 - \mu_1)^2(\mu_3 - \mu_2)^2}\\
& \quad - \dfrac{2C_{12}(\mu_1 - \mu_2)t^2 - 2(\mu_3 - \mu_1)^2(\mu_3 - \mu_2)tL_2(t) - C_{12}(\mu_3 - \mu_1)^2L_2(t)}{(\mu_2 - \mu_1)^2(\mu_3 - \mu_1)^2}\biggr], 
\end{align*}
that is
\begin{equation} \label{b1-1st-step}
b_1^2 = \dfrac{4(\mu_1 - \mu_2)\left((\mu_3 - \mu_1)^3 + (\mu_2 - \mu_1)^3 + (\mu_2 - \mu_3)^3\right)t^3 + \tilde{Q}_1(t)}{\det M \, (\mu_2 - \mu_1)^2(\mu_3 - \mu_1)^2(\mu_3 - \mu_2)^2},
\end{equation}
where $\tilde{Q}_1(t)$ is a polynomial of degree 2. Analogous computations provide
\begin{align}\label{b2-1st-step}
b_2^2 &= \dfrac{4(\mu_1 - \mu_2)\left((\mu_3 - \mu_2)^3 + (\mu_2 - \mu_1)^3 + (\mu_1 - \mu_3)^3\right)t^3 + \tilde{Q}_2(t)}{\det M \, (\mu_2 - \mu_1)^2(\mu_3 - \mu_1)^2(\mu_3 - \mu_2)^2},
\end{align}
where $\tilde{Q}_2(t)$ is also a polynomial of degree 2.

Now we will compute the determinant of $M$. Note that
\begin{align*}
\det M &= \dfrac{4\bigg((\mu_3 - \mu_2)^3 + (\mu_2 - \mu_1)^3\bigg)\bigg((\mu_3 - \mu_1)^3 + (\mu_2 - \mu_1)^3\bigg)t^2}{(\mu_1 - \mu_2)^4(\mu_1 - \mu_3)^2(\mu_2 - \mu_3)^2} \\
& \quad + \dfrac{2\bigg(C_{11}(\mu_3 - \mu_1)^3 + C_{11}(\mu_2 - \mu_1)^3 + C_{22}(\mu_3 - \mu_2)^3 + C_{22}(\mu_2 - \mu_1)^3\bigg)t + C_{11}C_{22}}{(\mu_1 - \mu_2)^4(\mu_1 - \mu_3)^2(\mu_2 - \mu_3)^2} \\
& \quad + \dfrac{4(\mu_1 - \mu_3)(\mu_3 - \mu_2)t^2}{(\mu_1 - \mu_2)^4} - \dfrac{2\bigg(C_{12}(\mu_3 - \mu_1) + C_{21}(\mu_3 - \mu_2)\bigg)t + C_{12}C_{21}}{(\mu_1 - \mu_2)^4} 
\end{align*}

\begin{align*}
&= \dfrac{4\left[\bigg((\mu_3 - \mu_2)^3 + (\mu_2 - \mu_1)^3\bigg)\bigg((\mu_3 - \mu_1)^3 + (\mu_2 - \mu_1)^3\bigg) - (\mu_3 - \mu_1)^3(\mu_3 - \mu_2)^3\right]t^2 +\tilde{L}(t)}{(\mu_1 - \mu_2)^4(\mu_1 - \mu_3)^2(\mu_2 - \mu_3)^2} . \\
\end{align*}
where $\tilde{L}(t)$ is a degree one polynomial in  $t$. Therefore, we conclude that $\det M$ is given by
\begin{align} \label{DetD-1st-step}
\det M &= \dfrac{4(\mu_2 - \mu_1)^3\bigg((\mu_3 - \mu_2)^3 + (\mu_3 - \mu_1)^3 + (\mu_2 - \mu_1)^3\bigg)t^2+\tilde{L}(t)}{(\mu_1 - \mu_2)^4(\mu_1 - \mu_3)^2(\mu_2 - \mu_3)^2}.
\end{align}

Now we will consider the coefficients of the highest degree terms of the polynomials given by the numerators of $b_1^2$, $b_2^2$ and $\det M$. We have
\begin{align*}
4\left((\mu_3 - \mu_1)^3 + (\mu_2 - \mu_1)^3 + (\mu_2 - \mu_3)^3\right) &= 12\mu_3^2(\mu_2 - \mu_1) - 12\mu_3(\mu_2^2 - \mu_1^2) + 8(\mu_2^3 - \mu_1^3) \\
& \quad- 12\mu_1\mu_2(\mu_2 - \mu_1)\\
&=(\mu_2 - \mu_1)\left(12\mu_3^2 - 12\mu_1\mu_3 - 12\mu_2\mu_3 + 8\mu_2^2 + 8\mu_1^2 - 4\mu_1\mu_2\right) \\
&=2(\mu_2 - \mu_1)\left(3(\mu_1 - \mu_3)^2 + 3(\mu_2 - \mu_3)^2 + (\mu_1 - \mu_2)^2\right)
\end{align*}
and similarly
\begin{align*}
4\left((\mu_3 - \mu_2)^3 + (\mu_3 - \mu_1)^3 + (\mu_2 - \mu_1)^3\right) &= 8(\mu_3^3 - \mu_1^3) - 12\mu_2(\mu_3^2 - \mu_1^2) + 12\mu_2^2(\mu_3 - \mu_1) \\
&\quad- 12\mu_1\mu_3(\mu_3 - \mu_1) \\
&=(\mu_3 - \mu_1)\left(8\mu_1^2 - 12\mu_1\mu_2 - 4\mu_1\mu_3 + 12\mu_2^2 - 12\mu_2\mu_3 + 8\mu_3^2\right) \\
&=2(\mu_3 - \mu_1)\left((\mu_1 - \mu_3)^2 + 3(\mu_1 - \mu_2)^2 + 3(\mu_2 - \mu_3)^2\right).
\end{align*}
\begin{align*}
(\mu_3 - \mu_2)^3 + (\mu_2 - \mu_1)^3 + (\mu_1 - \mu_3)^3  &= 3\left(\mu_3^2(\mu_1 - \mu_2) + \mu_3(\mu_2^2 - \mu_1^2) + \mu_{1}\mu_2(\mu_1 - \mu_2)\right) \\
&=3(\mu_1 - \mu_2)\left(\mu_3^2 - \mu_3(\mu_1 + \mu_2) + \mu_{1}\mu_2\right) \\
&=3(\mu_1 - \mu_2)(\mu_1 - \mu_3)(\mu_2 - \mu_3).
\end{align*}

In conclusion, we have
\begin{align}
b_1^2 &= \dfrac{-2(\mu_1 - \mu_2)^4\bigg(3(\mu_1 - \mu_3)^2 + 3(\mu_2 - \mu_3)^2 + (\mu_1 - \mu_2)^2\bigg)t^3 + Q_1(t)}{2(\mu_1 - \mu_3)(\mu_1 - \mu_2)^3\bigg((\mu_1 - \mu_3)^2 + 3(\mu_1 - \mu_2)^2 + 3(\mu_2 - \mu_3)^2\bigg)t^2 + L(t)}, \label{eqb_1^2} \\
b_2^2 &= \dfrac{12(\mu_1 - \mu_2)^4(\mu_1 - \mu_3)(\mu_2 - \mu_3)t^3 + Q_2(t)}{2(\mu_1 - \mu_3)(\mu_1 - \mu_2)^3\bigg((\mu_1 - \mu_3)^2 + 3(\mu_1 - \mu_2)^2 + 3(\mu_2 - \mu_3)^2\bigg)t^2 + L(t)},\label{eqb_2^2} 
\end{align}
where $Q_1(t)$ and $Q_2(t)$ are quadratic polynomials and $L(t)$ is a degree one polynomial in $t$. Finally, since $b_3^2 = 1-t-b_1^2-b_2^2$, a straightforward computation gives
\begin{align}
b_3^2 &= \dfrac{-2(\mu_2 - \mu_3)(\mu_1 - \mu_2)^3\bigg(3(\mu_1 - \mu_2)^2 + 3(\mu_1 - \mu_3)^2 + (\mu_2 - \mu_3)^2\bigg)t^3 + Q_3(t)}{2(\mu_1 - \mu_3)(\mu_1 - \mu_2)^3\bigg((\mu_1 - \mu_3)^2 + 3(\mu_1 - \mu_2)^2 + 3(\mu_2 - \mu_3)^2\bigg)t^2 + L(t)}, \label{eqb_3^2}
\end{align}
where $Q_3$ is a quadratic polynomials in $t$. 

\end{proof}

We are now in a position to prove our main lemma.

\begin{lemma}\label{lemma-characterization}
Let $\Sigma$ be a connected hypersurface of $\Qi^3 \times \Ri$ with three distinct constant principal curvatures. Then $\cos\theta$ is constant on $\Sigma$.
\end{lemma}

\begin{proof} Let $\Omega \subset \Sigma$ be an open subset as in Lemma \ref{lemma-equations}. We will show that $\cos\theta$ is constant on $\Omega$, such that, by continuity, $\cos\theta$ is constant on $\Sigma$. We give a proof by contradiction. So suppose that $\cos\theta$ is not constant. In particular $\cos\theta \neq 0$ on a dense open subset of $\Omega$ and, by Lemma~\ref{lemma-expressions-bi}, the same holds for $b_1$, $b_2$ and $b_3$. 

We will now calculate $e_1(b^2_1)$, $e_2(b^2_2)$ and  $e_3(b^2_3)$ using Lemma \ref{lemma-equations} and Lemma \ref{lemma-expressions-bi} to find polynomial identities in the variable $t=\cos^2\theta$. Since $t$ is not constant, the polynomials have to be identically zero. Let us start with $e_1(b^2_1)$. On the one hand, using \eqref{b_i^2} in the first equality and \eqref{cos-deriv} in the second, we have
$$ e_1(b_1^2) = \frac{p_1'(t)q(t)-p_1(t)q'(t)}{q(t)^2} \, e_1(t) = -2\mu_1b_1\cos\theta \, \frac{p_1'(t)q(t)-p_1(t)q'(t)}{q(t)^2}. $$
On the other hand, using \eqref{en-bn}, we have
$$ e_1(b_1^2) = 2b_1e_1(b_1) = 2 b_1 \cos\theta \left[ \mu_1 - \varepsilon \left( \frac{b_2^2}{\mu_2-\mu_1} + \frac{b_3^2}{\mu_3-\mu_1} \right) \right]. $$
By comparing the two expressions above and using \eqref{b_i^2}, we obtain
\begin{equation*}\label{e1-b1}
\varepsilon q(t) [(\mu_3 \!-\! \mu_1)p_2(t) \!+\! (\mu_2 \!-\! \mu_1)p_3(t)] \!-\! \mu_1(\mu_2 \!-\! \mu_1)(\mu_3 \!-\! \mu_1)[(p_1'(t)q(t) \!-\! p_1(t)q'(t)) \!+\! q^2(t)] = 0.
\end{equation*}
As explained before, since $t$ is not constant, the polynomial has to vanish identically. In particular, the leading coefficient, i.e., the coefficient of $t^5$, has to vanish. Note that that $t^5$ only appears in $\varepsilon q(t) [(\mu_3 \!-\! \mu_1)p_2(t) \!+\! (\mu_2 \!-\! \mu_1)p_3(t)]$ and its coefficient is
\begin{equation}\label{coefficient-1}
-8\varepsilon(3\mu_1 \!-\! 2\mu_2 \!-\! \mu_3)(\mu_1 \!-\! \mu_2)^7(\mu_2 \!-\! \mu_3)^2(\mu_1\!-\! \mu_3)[(\mu_1 \!-\! \mu_3)^2 \!+\! 3(\mu_1 \!-\! \mu_2)^2 \!+\! 3(\mu_2 \!-\! \mu_3)^2].
\end{equation}

Now let us consider $e_2(b^2_2)$. In an analogous way as before, we obtain
\begin{equation*}\label{e2-b2}
\varepsilon q(t) [(\mu_3 \!-\! \mu_2)p_1(t) \!+\! (\mu_1 \!-\! \mu_2)p_3(t)] \!-\! \mu_2(\mu_1 \!-\! \mu_2)(\mu_3 \!-\! \mu_2)[p_2'(t)q(t) \!-\! p_2(t)q'(t)) \!+\! q^2(t)] = 0
\end{equation*}
and hence the leading coefficient, 
\begin{equation}\label{coefficient-2}
-8\varepsilon(\mu_1 \!-\! 2\mu_2 \!+\! \mu_3)(\mu_1 \!-\! \mu_2)^7(\mu_1 \!-\! \mu_3)^2(\mu_2 \!-\! \mu_3)[(\mu_1 \!-\! \mu_3)^2 \!+\! 3(\mu_1 \!-\! \mu_2)^2 \!+\! 3(\mu_2 \!-\! \mu_3)^2],
\end{equation}
must vanish. Finally, we consider $e_3(b^2_3)$. Since
\begin{equation*}\label{e3-b3}
\varepsilon q(t) [(\mu_1 \!-\! \mu_3)p_2(t) \!+\! (\mu_2 \!-\! \mu_3)p_1(t)] \!-\! \mu_3(\mu_1 \!-\! \mu_3)(\mu_2 \!-\! \mu_3)[p_3'(t)q(t) \!-\! p_3(t)q'(t) \!+\! q(t)^2] = 0,
\end{equation*}
the leading coefficient,
\begin{equation}\label{coefficient-3}
8\varepsilon(\mu_1 \!+\! 2\mu_2 \!-\! 3\mu_3)(\mu_1 \!-\! \mu_2)^8(\mu_1 \!-\! \mu_3)(\mu_2 \!-\! \mu_3)[(\mu_1 \!-\! \mu_3)^2 \!+\! 3(\mu_1 \!-\! \mu_2)^2 \!+\! 3(\mu_2 \!-\! \mu_3)^2],
\end{equation}
must vanish. Since the expressions \eqref{coefficient-1}, \eqref{coefficient-2} and \eqref{coefficient-3} have to vanish and $\mu_1$, $\mu_2$ and $\mu_3$ are distinct, we obtain the following homogeneous linear system:
\begin{align}
& 3\mu_1 - 2\mu_2 - \mu_3 = 0, \label{systemequations1}\\
& \mu_1 - 2\mu_2 + \mu_3 = 0, \label{systemequations2}\\
& \mu_1 + 2\mu_2 - 3\mu_3 = 0.\label{systemequations3}
\end{align}
From \eqref{systemequations1} and \eqref{systemequations2}, we obtain that $\mu_1=\mu_2=\mu$ , and therefore, from \eqref{systemequations3}, we conclude that $\mu_3=\mu$, which contradicts the fact that the principal curvatures are distinct. This concludes the proof of the lemma.
\end{proof}

\section{Proof of the main results}

\begin{proof}[Proof of Theorem \ref{theorem-classification}]
Let $\Sigma$ be a hypersurface of $\Qi^3 \times \R$ with three distinct constant principal curvatures. By Lemma \ref{lemma-characterization}, $\cos\theta$ is constant, such that also $\|T\|$ is constant, by the decomposition $\partial_{t} = T + \cos\theta\,N$. Since the principal curvatures are distinct, $T \neq 0$, otherwise $\Sigma$ is part of a slice $\Qi^3 \times \{t_{0}\}$, which is totally geodesic. Therefore, it follows from \eqref{Xcos} that $T$ is a principal direction, with associated principal curvature $0$. Suppose that $\cos\theta \neq 0$. Then, from \cite[Theorem 4.1]{Rosa-and-Eliane}, it follows that $\varepsilon = -1$ and $\Sigma$ is locally parametrized by $f(p,s) = h_{s}(p) + Bs\,\partial_t$ for some $B \in \Ri$, $B > 0$, where $h_{s}$ is a family of parallel horospheres in $\Hi^3$. Moreover, the non-zero principal curvatures are both equal, which is a contradiction. Therefore $\cos\theta = 0$ and $\Sigma$ is an open part of a vertical cylinder. Since $\Sigma$ has three distinct constant principal curvatures, the principal curvatures of the base surface have to be constant, non-zero and distinct. The classification goes back to Cartan \cite{isoCartan}, see also Theorem $1$ in ~\cite{ojm/1200693241}.
\end{proof}

\begin{proof}[Proof of Theorem \ref{Corollary}]
Corollary 3.4 in \cite{Rosa-and-Eliane} states that a hypersurface of $\Qi^n\times\R$ for which $\cos\theta$ is a constant different from $1$ is isoparametric if and only if it has constant principal curvatures. The result thus follows from proving the constancy of $\cos\theta$ and the observation that slices are isoparametric. If the number of distinct principal curvatures $g$ equals $1$ or $2$, this is already established in \cite[Theorem 6.1]{Rosa-and-Eliane}. For $g=3$, it follows from Lemma \ref{lemma-characterization}.
\end{proof}

\begin{proof}[Proof of Theorem \ref{corollary-2}] It is well-known that an extrinsically homogeneous hypersurface has constant principal curvatures and is isoparametric (see for example \cite[Proposition 2.10]{notas-miguel}). By Theorem \ref{Corollary} it is enough to search such hypersurfaces among those with constant principal curvatures. The hypersurfaces given in items (i), (ii)(a) and (ii)(b) are provided by the classification theorem \cite[Theorem 6.1]{Rosa-and-Eliane} of hypersufaces with $g\in\{1,2\}$ constant principal curvatures, whereas (ii)(c) and (ii)(d) are given by Theorem \ref{theorem-classification}. All these hypersurfaces are obviously extrinsically homogeneous. The hypersurface of item (iii) is also provided by \cite[Theorem 6.1]{Rosa-and-Eliane} and has $g=2$. In order to show that it is extrinsically homogeneous, we use the same argument as the one used in \cite{dominguez-manzano} to show that the parabolic helicoid in $\mathbb{H}^2 \times \mathbb{R}$ is extrinsically homogeneous. In fact, as we pointed out in the paragraph right after the statement of Theorem \ref{corollary-2}, the hypersurface of item (iii) is given by the graph $\Sigma = \{(u,\,v,\,w,\,B \log(w)) \,|\, u,v \in \R, \, w>0 \} \subset \mathbb{H}^3 \times \mathbb{R}$, which is invariant by the 2-parameters group of isometries $(u,v,w,t) \mapsto (u+\lambda_1, v+\lambda_2, w, t)$ and the 1-parameter group of isometries $(u,v,w,t) \mapsto (e^{\lambda} u,\,e^{\lambda} v,\,e^{\lambda} w,\, t-B \lambda)$. Therefore, $\Sigma$ is extrinsically homogeneous. \end{proof}

\bibliographystyle{abbrv}
\bibliography{refs}

\end{document}